\documentclass[11pt]{article}
\usepackage[margin=1in]{geometry}
\usepackage{cite}
\makeatletter
\let\NAT@parse\undefined
\makeatother
\usepackage{hyperref}
\usepackage{amsmath,amssymb,amsfonts,amsthm}
\usepackage{algorithmic}
\usepackage{graphicx}
\usepackage{textcomp}
\usepackage[caption=false,font=normalsize,labelfont=sf,textfont=sf]{subfig}
\usepackage{booktabs}
\usepackage{multirow}
\usepackage{easyReview}
\usepackage[ruled,linesnumbered]{algorithm2e}
\newtheorem{assumption}{Assumption}
\newtheorem{remark}{Remark}

\newtheorem{theorem}{Theorem}
\newtheorem{corollary}{Corollary}
\newtheorem{lemma}{Lemma}

\newcommand{\IEEEmembership}[1]{}
\newcommand{\IEEEPARstart}[2]{#1#2}
\newenvironment{IEEEkeywords}
  {\par\smallskip\noindent\textbf{Keywords:}\ }
  {\par\smallskip}

\def\BibTeX{{\rm B\kern-.05em{\sc i\kern-.025em b}\kern-.08em
    T\kern-.1667em\lower.7ex\hbox{E}\kern-.125emX}}
\begin{document}
\title{Performance guaranteed MPC Policy Approximation via Cost Guided Learning}
\author{	Chenchen Zhou, Yi Cao, \IEEEmembership{Senior Member, IEEE}, and Shuang-hua Yang, \IEEEmembership{Senior Member, IEEE}
	\thanks{This work was supported by the Research Funds of Institute of Zhejiang University-Quzhou and the Applied Basic Research Fund of the Institute of Zhejiang University-Quzhou (IZQ2022KJ1003) }
	\thanks{ Chenchen Zhou, Yi Cao and Shuang-hua Yang are with the College of Chemical and Biological Engineering, Zhejiang University, Hangzhou 310027, China and also with the Institute of Zhejiang University-Quzhou, Quzhou 324000, China (e-mail: yangsh@zju.edu.cn).}
	\thanks{Shuang-hua Yang is also with the Department of Computer Science, University of Reading,	Reading RG6 6AH, U.K. (e-mail: shuang-hua.yang@reading.ac.uk).}
}

\maketitle

\begin{abstract}
	Model predictive control (MPC) is widely used in industries but implementing it poses challenges due to  hardware or time constraints. A promising solution is to approximate the MPC policy using function approximators like neural networks. Existing methods focus on minimizing the error between the approximators outputs and the MPC optimal control actions on training data, which is called error guided learning approach in this paper. However, the goals of control law design is not to minimize the fitting error but to minimize the operation cost.
	This paper proposes a novel cost-guided learning approach that utilizes the cost sensitivity information from the MPC problem to directly minimize the loss in closed-loop performance. A theoretical analysis shows cost-guided learning provides tighter guarantees on optimality loss compared to traditional error-guided learning. Experiments on a continuous stirred tank reactor (CSTR) benchmark demonstrate that the proposed technique results in approximate MPC policies that achieve substantially better closed-loop performance. This work makes an important contribution by connecting the fitting errors with operational objectives, overcoming key limitations of existing approximation methods. The core idea could be applied more broadly for data-driven control.
\end{abstract}

\begin{IEEEkeywords}
	Machine learning, Predictive control for nonlinear systems, Optimization	
\end{IEEEkeywords}

\section{Introduction}
\label{sec:introduction}
\subsection{Pre-computed control policies}
\IEEEPARstart{M}{odel} Predictive Control (MPC) has been successfully applied in various fields\cite{darby2012a} to optimize dynamic system performance. 
However, implementing MPC faces challenges due to computational load and limited memory.
Methods like pre-computing optimal policies offline emerged to enable online usage without re-optimization. \textit{Explicit MPC} pioneered expressing MPC as piecewise-affine functions on polytopes \cite{bemporad2002}.  
But the polytopic regions grow exponentially with variables and constraints, limiting its usability for nonlinear systems\cite{krishnamoorthy2021}. 
An alternative is approximating the optimal policy with parametric approximators like neural networks, studied in \textit{learning from demonstrations}. 
The idea of using neural networks as function approximators for nonlinear MPC (NMPC) feedback laws was proposed, as early as 1995\cite{parisini1995}.  With recent deep learning advances, interest surged in approximating policies, see the tutorial \cite{mesbahFusionMachineLearning2022}. 

Research has focused on understanding safety and performance of approximate policies \cite{hertneck2018a,paulson2020a,karg2021}. 
Although these studies are important for advancing MPC policy approximation, there is a research gap - the training objective for approximate policies attempt to imitate an expert policy, rather than the ultimate goal of good performance when implemented online.
Performance is measured by operational cost, time, energy use, tracking error, etc. as defined in the MPC problem.
\subsection{MPC policy approximation by error guided learning}
Consider a nonlinear nonlinear MPC problem $\mathcal{P}(x(t))$ can be formulated as follows: 
\begin{subequations}
	\label{eq:MPC}
	\begin{align}
		&\min _{ u(\cdot)}   \sum_{k=0}^{N-1} \ell(x(k ), u(k))+\ell_f(x(N)) \label{eq:MPC_cost}\\
		&\text { s.t. }  x(k+1)=f(x(k), u(k)) \quad \forall k \in \mathbb{I}_{0: N-1} \label{eq:MPC_mc} \\
		&x(k) \in \mathcal{X}, \quad u(k) \in \mathcal{U} ,\quad x(N) \in \mathcal{X}_f, \quad x(0)=x(t)\label{eq:MPC_pc}
	\end{align}
\end{subequations}
where $x(t) \in \mathbb{R}^{n_x}$ and $u(t) \in \mathbb{R}^{n_u}$ denote the states and control inputs at time $t$, respectively. The mapping $f: \mathbb{R}^{n_x} \times \mathbb{R}^{n_u}\rightarrow \mathbb{R}^{n_x}$ represents the discrete-time plant model. $\ell: \mathbb{R}^{n_x} \times \mathbb{R}^{n_u} \rightarrow \mathbb{R}$ represents the stage cost, which can be either a tracking or economic objective, and $\ell_f: \mathbb{R}^{n_x} \rightarrow \mathbb{R}$ represents the terminal cost.
The parameter $N$ denotes the prediction horizon length, while \eqref{eq:MPC_pc} represents the constraints. In the traditional MPC paradigm, the optimization problem $\mathcal{P}(x(t))$ is solved at each time step $t$ by taking the current state $x(t)$ as initial condition, and the resulting optimal input $u(t) = u(0)$ is applied to the plant in a receding horizon manner. This implies a feedback control policy $\pi_{\mathrm{MPC}}: \mathbb{R}^{n_x} \rightarrow \mathbb{R}^{n_u}$ can be defined as:
	$
	u^*(t) = \pi_{\mathrm{MPC}}(x(t))
	$

The typical framework of approximating MPC policy is as follows \cite{parisini1995,paulson2020a}: 
 To approximate the MPC control law $\pi_{\mathrm{MPC}}(x(t))$, the feasible state space $\mathcal{X}$ is sampled, generating a set of randomly selected initial states $\{x_i\}_{i=1}^{N_s}$. For each initial state $x_i$, the MPC problem $\mathcal{P}(x_i)$ is solved to obtain the corresponding optimal input $u_i^* = \pi_{\mathrm{MPC}}(x_i)$. Using the collected data samples $\mathcal{D} = \{\left(x_i, u_i^*\right)\}_{i=1}^{N_s}$, a parametric function $\pi_{\text{approx}}(x ; \theta)$ is trained, parameterized by $\theta$, to minimize the mean squared error:
\begin{equation}
	\label{eq:mse}
	L_{\text{mse}}(\theta)=\frac{1}{N_s} \sum_{i=1}^{N_s}\left\|\pi_{\text {approx }}\left(x_i ; \theta\right)-u_i^*\right\|^2 ,
\end{equation}
where $ L_{\text{mse}} $ denotes the loss function in the form of sum of square errors, and in this paper, we named it \textit{fitting errors guided loss function}.
A perfect approximation $ \pi_{\text {approx }}\left(x_i\right)=\pi_{\mathrm{MPC}}\left(x_i\right) $ would minimize the operation cost. But fitting errors always exist, and it is unclear how they relate to the operation cost. There has been limited focus on directly using cost $J$ to guide policy learning.
There are some articles that analyze the impact of fitting errors to the operation cost \cite{karg2021}. 
However, those papers focus on evaluate the trained policy, not using the operation cost to guide the policy training.

In summary, most of existing MPC policy approximation methods aim to minimize the fitting errors to the optimal policy, however, techniques that directly account for operational costs to guide neural network training remain scarce \cite{chen2022a}, with the connection between fitting error and operational costs still murky. This limitation might lead to policies that may have small errors but still perform poorly when implemented.
To the end, the main contributions of this paper are summarized as follows:
\begin{enumerate}
	\item This paper proposes a novel cost-guided learning approach for MPC policy approximation. In contrast to traditional methods that minimize fitting errors, this technique directly minimizes the loss in operational cost caused by fitting errors.
	\item  A theoretical analysis shows cost-guided learning provides tighter guarantees on optimality loss compared to error-guided learning. The proposed method utilizes the sensitivity information from the MPC problem to focus learning on high-cost areas of the state space.
\end{enumerate}
The article continues as follows.
In Section II, the cost-guided learning technique is proposed, and it is shown that this technique can lead to a tighter upper bound on the optimality gap than the error-guided learning technique for training approximate policies. Furthermore, an example is used to illustrate the proposed approach in Section III. The article is concluded in Section IV.
\section{Objective guided deep learning}

\begin{figure*}[th]
	\centering
	\subfloat[]
	{
		\includegraphics[width=0.27\textwidth]{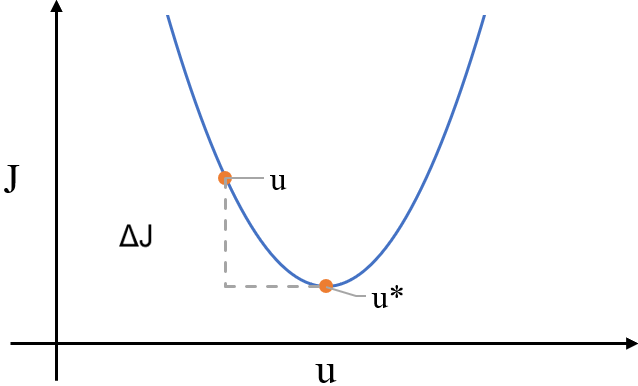}
		\label{fig:sharp}
	}
	\subfloat[]
	{
		\includegraphics[width=0.27\textwidth]{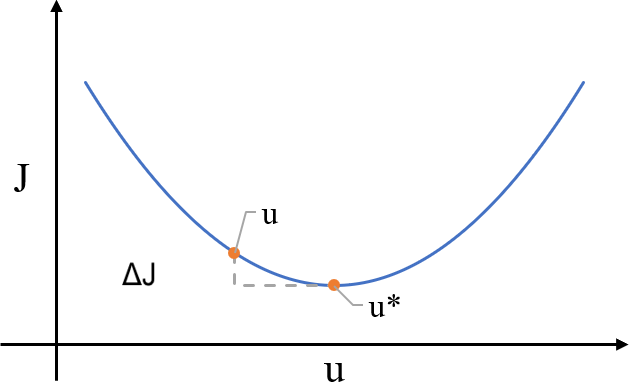}
		\label{fig:flat}
	}
	\subfloat[]
	{
		\includegraphics[width=0.27\textwidth]{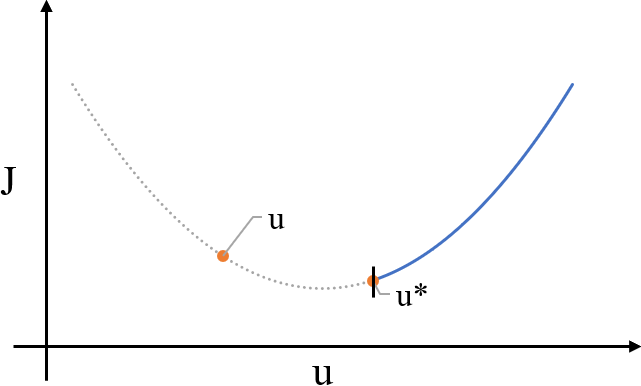}
		\label{fig:constrained}
	}
	\caption{Three types of optimums: (a) Sharp optimum; (b) flat optimum; (c) Constrained optimum. $J$ stands for operational cost, $u$ denotes the control input, $u^*$ stands for the optimal control input and $\Delta J$ is the losses caused by non-optimal inputs. 
	The blue curve is the $J$ where the constraints would be satisfied by that $u$. The gray dotted curve is $J$ where the constraints would be violated by that $u$. Black lines denote constraints.}
	\label{fig:optsurface}
\end{figure*}

The fitting errors of the trained policy can vary for different system states $x(t)$. 
Common MPC approximations use the mean squared error \eqref{eq:mse} as the loss function. This weights all errors equally.
However, approximations deviate from the exact solution, impacting system behavior. 
To provide further clarification, we can, as illustrated in Fig.~\ref{fig:optsurface}, distinguish between three classes of optimums when it comes to the actual implementation:
\begin{enumerate}
	\item \textbf{Unconstrained sharp optimum} Fig. \ref{fig:sharp}: In this case, the operation cost is sensitive to the value of control input.   Non-optimal $u$ incur high losses $\Delta J$.
	\item \textbf{Unconstrained flat optimum} Fig. \ref{fig:flat}: In this case, the operation cost is insensitive to the value of control input.  Non-optimal $u$ cause smaller $\Delta J$ than in the sharp optimum.
	\item \textbf{Constrained optimum} Fig. \ref{fig:constrained}: The more serious problem for implementation may occur when the problem is constrained. In this case, the optimal solution lies on the boundary of constraints. For example, optimal operation may require running the system at maximum tolerable load. Exceeding this load capacity can cause system failure. Here, fitting errors may result in an infeasible operation, potentially leading to safety accidents.
\end{enumerate}
Fig. \ref{fig:sharp} and Fig. \ref{fig:flat} illustrate that the same fitting errors may lead to different degrees of loss of optimality due to the sensitivity of the solutions of different $\mathcal{P}(x(t))$. 
For constrained problems, the error direction also affects constraint satisfaction. In some scenarios, the fitting errors might cause the system to violate certain constraints that were initially satisfied with the exact solution. 
There are some paper focus on this problem \cite{hertneck2018a,paulson2020a,karg2021}.
As mentioned previously, the sensitivity information from the MPC problem solution can be utilized to guide the training of the policy. In this section, we propose a cost-guided learning approach to approximate the MPC policy. 

\subsection{Control policy evaluation indicator}
For the sake of simplicity, we rewrite the MPC problem \eqref{eq:MPC} into a standard parametric nonlinear programming (NLP) problem $\mathcal{P}(p)$ of the following form:
\begin{subequations}
	\label{eq:NLP}
	\begin{align}
		&\min _{\mathbf{w}} J(\mathbf{w}, p) \label{eq:NLP_ec} \\
		&\text { s.t. } c(\mathbf{w}, p)=0, \qquad
		g(\mathbf{w}, p) \leq 0  \label{eq:NLP_cost} 
	\end{align}
\end{subequations}
where the initial state denoted as $p = x(0\vert t) = x(t)$. The decision variables are represented by $\mathbf{w} = [u(0 \vert t), \ldots, u(N-1 \vert t)]^{\top}$. The cost function \eqref{eq:MPC_cost} is now represented by $J$. Similarly, the system equations \eqref{eq:MPC_mc} are denoted by $c$, while the state and input constraints and the terminal constraint \eqref{eq:MPC_pc} are collectively referred to as $g$.  

As shown in Fig. \ref{fig:optsurface}, the operational cost $J$ alone is not a good indicator to evaluate control policy performance. In the constrained optimum case, inputs violating constraints can incur lower $J$. However, this does not mean such policies with constraint violations are better. Next, we introduce the Lagrangian function to evaluate control policies. 
The Lagrangian function of problem \eqref{eq:NLP} can be expressed as follows:
\begin{equation}\label{eq:lag}
	\mathcal{L}(\mathbf{w}, \lambda, \mu, p):=J(\mathbf{w}, p)+\lambda^{\top} c(\mathbf{w}, p)+\mu^{\top} g(\mathbf{w}, p)
\end{equation}
where $\lambda$ and $\mu$ are the Lagrangian multipliers of $c$ and $g$ respectively. 
For the inequality constraints, we define $g_{\mathbb{A}}(\mathbf{w}, p) \subseteq g(\mathbf{w}, p)$ as the set of active inequality constraints (i.e. $g_{\mathbb{A}}(\mathbf{w}, p) = 0$).
The Karush Kuhn Tucker (KKT) condition for this problem is given by
\begin{equation}\label{eq:KKT}
	\begin{aligned}%
		\psi(s(p),p) =:& 
		\begin{bmatrix}
			\mathcal{L}_{\mathbf{w}}(\mathbf{w}, \lambda, \mu, p)\\
			c(\mathbf{w}, p) \\
			g_{\mathbb{A}}(\mathbf{w}, p) 
		\end{bmatrix} =0 \\
		 g_{\bar{\mathbb{A}}}(\mathbf{w}, p)& < 0 , \quad \mu_{\bar{\mathbb{A}}} = 0 , \quad \mu_{\mathbb{A}} > 0 
	\end{aligned}
\end{equation}
For any point $\mathbf{s}^*(p):=\left[\mathbf{w}^*, \lambda^*, \mu^*\right]^{\top}$ which satisfies the KKT conditions \eqref{eq:KKT} for given $p$, it is called KKT point.
\begin{assumption}
	\label{ass:4}
	$J(\cdot,\cdot)$, $c(\cdot,\cdot)$, and $g(\cdot,\cdot)$ of the parametric NLP problem $\mathcal{P}(p)$ are twice continuously differentiable in a neighborhood of the KKT point $s^*(p)$. 
\end{assumption}
\begin{assumption}
	\label{ass:5}
	Linear independence constraint qualification (LICQ), Strong second-order sufficient conditions (SSOSC) and strict complementarity (SC) hold for the solution $s^*(p)$
\end{assumption}
These assumptions imply that KKT point $\mathbf{s}^*(p)$ is a unique local optimum of $\mathcal{P}(p)$ and is also a saddle point of  $\mathcal{L}(\mathbf{w}, \lambda, \mu, p)$ in
 the neighborhood of $\mathbf{s}^*(p)$, which is 
 $$
 \begin{aligned}
 	\mathcal{L}(\mathbf{w}^*, \lambda, \mu, p) &\leq \mathcal{L}(\mathbf{w}^*, \lambda^*, \mu^*, p) \\
 	 &= J((\mathbf{w}^*, p)) \leq \mathcal{L}(\mathbf{w}, \lambda^*, \mu^*, p)
 \end{aligned}
 $$
The larger Lagrangian values can indicate worse control law performance.
So we introduce the discrepancy of the Lagrangian values as control policy evaluation indicator:
\begin{equation}
	\label{eq:1lag}
	\Delta{\mathcal{L}}(u,p)=                                                                
	{\mathcal{L}}([u,\mathbf{\bar{w}}^*],\lambda^{*},\mu^{*},p)	-{\mathcal{L}}(\mathbf{w}^{*},\lambda^{*},\mu^{*},p)		
\end{equation}
where $\mathbf{w}^*=[ u_0^*,\mathbf{\bar{w}}^*] $ is the decision variable of the primal problem and $u_0^*$ denotes the optimal input at the current time, $\mathbf{\bar{w}}^*$ represents the rest decision variables except $u_0^*$.

\begin{assumption}
	When $\left\|u-u^*\right\|^2  \leq  \rho$ the active constraints $g_{\mathbb{A}}$ remain unchanged.
\end{assumption}
It is important to note the LICQ condition ensures active constraints do not change within a neighborhood of the optimum. Therefore, if $\|u-u^*\|$ is sufficiently small, $\mathcal{L}(u,p)$ can account for both constraint satisfaction and operational cost. 
So Lagrangian based control policy evaluation indicator could be expressed as follows:
\begin{equation}
	L_{\text{lag}}(\theta):=\dfrac{1}{N_s}\sum_{i=1}^{N_s}
	\left\| \Delta{\mathcal{L}}_i(\pi_{\text{approx}}(x_{i};\theta)) \right\|^{2} .
\end{equation}
Some literature reports that utilizing Lagrangian values as a loss is effective for linear systems \cite{chen2022a}. However, for nonlinear systems, it is computationally costly since the Lagrangian incorporates the nonlinear model. This constrains large-scale implementation.



\subsection{Local approximation of the Lagrangian function}
For each $ p $, the  Lagrangian function \eqref{eq:lag} is approximated locally by a quadratic function around the optimal point $ \mathbf{s}^*(p) $. Introducing deviation variables $\Delta u = u- u^*$, a Taylor expansion around the optimal point $\mathbf{s}(p)^*$ gives
\begin{equation}
	\label{eq:2lag}
	\Delta{\mathcal{L}}_p(u)\approx \Delta{\hat{\mathcal{L}}}_p(u)= \mathcal{L}_{u}(p)^{\top}\Delta u+ {\Delta u}^{\top} \mathcal{L}_{uu}(p)  \Delta u  ,
\end{equation}
where $$ \mathcal{L}_{u}(p)= \frac{\partial{\mathcal{L}}}{\partial{u}} + \frac{\partial{\mathcal{L}}}{\partial{\mathbf{\bar{w}_{c}}}} \frac{\partial{c}}{\partial{u}}  $$ 
and 
$$\begin{aligned}
	\mathcal{L}_{uu}(p) 
	=\frac{\partial{\mathcal{L}_{u}(p)}}{\partial{u}} + \frac{\partial{\mathcal{L}_{u}(p)}}{\partial{\mathbf{\bar{w}_{c}}}} \frac{\partial{c}}{\partial{u}}
\end{aligned}  .$$ 
Here $\mathbf{\bar{w}_c}$ denotes the $n_c + n_{g_{\mathbb{A}}}$ variables related to constraints chosen from $\mathbf{\bar{w}}$, \textit{i.e.} $\mathbf{\bar{w}_c}$ denotes all equality constraints and active constraints. $n_c $ and $n_{g_{\mathbb{A}}}$ respectively represent the number of equality constraints and active inequality constraints. 
Because $\mathbf{s}^*(p) $ is the optimal point , and it satisfies the Karush-Kuhn-Tucker (KKT) Conditions, so $\nabla \mathcal{L}(\mathbf{s}(p)^*,p)=0$. 

Using the quadratic cost function, it can be shown that the Lagrangian form of loss function can be expressed as
\begin{equation}
	L_{\text{w}}(\theta) = \dfrac{1}{N_s}\sum_{i=1}^{N_s}\left| {\Delta\pi_{\text{approx}}(x_{i};\theta)}^{\top}  {\mathcal{L}_{uu}(p_i)} {\Delta\pi_{\text{approx}}(x_{i};\theta)} \right| .
\end{equation}
where $\Delta\pi_{\text{approx}}(x_{i};\theta) = \pi_{\text{approx}}(x_{i};\theta)-u_i^*$.
In the remaining of this paper, $L_{\text{w}}$ is referred to as \textit{cost guided loss function}.

$\|  {\mathcal{L}_{uu}(p_i)}\|$ might be much larger or smaller than $\|\pi_{\text{approx}}(x_{i};\theta)- u_i^*\|$, which will make the training fail due to this numerical reason. The Lagrangian form of loss function is modified to 
\begin{equation}
	L_{\text{w}}(\theta) = \dfrac{1}{N_s}\sum_{i=1}^{N_s} \dfrac{1}{\gamma }\left| {\Delta\pi_{\text{approx}}(x_{i};\theta)}^{\top}  {\mathcal{L}_{uu}(p_i)} {\Delta\pi_{\text{approx}}(x_{i};\theta)} \right|  ,
\end{equation}
where $\gamma = \sup_{i=1,2,\dots,N_s} \{ \|{\mathcal{L}_{uu}(p_i)}\| \}$.

\begin{remark}
	The formulation is the same as weighted regression.
	Traditional weighted regression is weighted by the inverse of the covariance matrix of the observations. 
	However, here, we do not focus on the heteroscedasticity (the unequal variance of observations) but on control system cost.
	
\end{remark}

\begin{algorithm}[h]
	\caption{Cost guided learning based MPC policy approximation .}
	\label{alg:ampcp_cg}
	\KwIn{$\mathcal{P}(x), \mathcal{X}, \mathcal{D}=\emptyset, N_p, N_s=0$}
	\KwOut{$\pi_{\text {approx }}\left(x ; \hat{\theta}_{\text{w}}\right)$}
	\BlankLine
	\For{$i$ in $1,2,...,N_p$}
	{ 
		Sample $x_i \in \mathcal{X}$ 
		
		$ s_i^*(x_i),flag  \leftarrow \operatorname{Solve} \mathcal{P}\left(x_i\right)$ 
		
		\If{flag == succeeded}{
			Extract $u_{i}^{*}$ from the primal-dual solution vector $\mathbf{s}^{*}(x_{i})$ 
			
			Compute the hessian matrix $ {\mathcal{L}_{uu}(p_i)} $ of Lagrangian function \eqref{eq:lag}, evaluated at $p_{i}^{*}$
			
			$ \mathcal{D} \leftarrow \mathcal{D} \cup\left\{\left(x_i, u_i^*\right)\right\}$
			
			$N_s = N_s+1$
		}
		
	}
	$\hat{\theta}_{\text{w}} \leftarrow \arg \min_\theta \frac{1}{N_s} \sum_{i=1}^{N_s}\left| {\Delta \pi_{\text{approx}}(x_{j};\theta)}^{\top} {\mathcal{L}_{uu}(p_i)} {\Delta \pi_{\text{approx}}(x_{j};\theta)} \right|$
\end{algorithm}
The pseudocode for the proposed cost guide learning approach for MPC policy approximation is summarized in Algorithm \ref{alg:ampcp_cg}. 
Additionally, it is worth noting that when the solved MPC is infeasible, that solution is not included in the dataset used, as shown in the algorithm. 

\subsection{Upper bound of the loss of optimality}
Next, we will analyze the relationship between error-guided loss functions and cost-guided loss functions concerning the optimality loss. 
We will first state Theorem \ref{the:lw} which shows cost-guided loss provides upper bound on true loss.
\begin{assumption}
	\label{ass:lJ}
	For any point $ \mathbf{s} \in \mathbf{S}({s}^*(p),\rho)=\{s:\|s-s^*\|\leq\rho\} $, $ \rho>0 $,
	the Lagrangian function $ \mathcal{L}(\cdot,p) $ is a once-differentiable function with Lipschitz Jacobian . And the local Lipschitz constant is $M_{\text{J}}(p)>0$, that is 
	$$\|\nabla \mathcal{L}(x,p)-\nabla \mathcal{L}(y,p)\|\leq M_{\text{J}}(p)\|x-y\|$$
\end{assumption} 
\begin{assumption}
	\label{ass:lH}
	For any point $ \mathbf{s} \in \mathbf{S}({s}^*(p),\rho)=\{s:\|s-s^*\|\leq\rho\} $, $ \rho>0 $,
	the Lagrangian function $ \mathcal{L}(\cdot,p) $ is a twice-differentiable function with Lipschitz Hessian. And the local Lipschitz constant is $M_{\text{H}}(p)>0$, respectively, that is 
	$$\|\nabla^2 \mathcal{L}(x,p)-\nabla^2 \mathcal{L}(y,p)\|\leq M_{\text{H}}(p)\|x-y\|$$
\end{assumption}
Assumption \ref{ass:lJ} and Assumption \ref{ass:lH} are commonly used in analyses of nonlinear optimization algorithms, such as in Newton's method \cite{nesterovCubicRegularizationNewton2006}. 
Assumption \ref{ass:lH} is made to ensure the bounded error between the second-order approximation of the Lagrangian and the Lagrangian itself. Assumption  \ref{ass:lJ} guarantees a bounded induced norm of the Lagrangian Hessian. Next, we will present two lemmas based on these assumptions.
\begin{lemma} [\cite{nesterovCubicRegularizationNewton2006} Lemma 1]
	\label{lem:cov1}
	If a twice differentiable function $ f(x) $, and the Hessian of function $ f $ is Lipschitz continuous on $\mathcal{F}$:
	$$
	\|\nabla^2 f(x)-\nabla^2 f(y)\| \leq M\|x-y\|, \forall x,y \in \mathcal{F}
	$$
	for some $ M> 0 $, then for any $ x $ and $ y $ from $ \mathcal{F} $, 	
	$$
	\|\nabla f(x)-\nabla f(y)-\nabla^2 f(y)(x-y)\|\leq{\frac{1}{2}} M\|x-y\|^{2}, 
	$$	
	$$
	\begin{aligned}
		| f(x)-f(y)-\nabla f(y)^{\top}(x-y) -  \frac{1}{2}(x-y)^{\top} \nabla^2 f(y)(x-y) |\\
		\leq{\frac{M}{6}}\|x-y\|^{3}. 
	\end{aligned}
	$$
\end{lemma}
\begin{proof}	
	$$
	\begin{array}{l}
		\|\nabla f(y)-\nabla f(x)-\nabla^2 f(x)(y-x)\|\\
		=\|\int_{0}^{1}[\nabla^2 f(x+\tau(y-x))-\nabla^2 f(x)](y-x)d\tau \| \\
		\le M\|y-x\|^{2}\displaystyle{\int_{0}^{1}}\tau d\tau 
		={\frac{1}{2}}M\|y-x\|^{2}.
	\end{array} 
	$$
	Therefore,
	$$
	\begin{array}{l}
		|f(y)-f(x)-\nabla f(x)^{\top}(y-x) -  \frac{1}{2}(y-x)^{\top} \nabla^2 f(x)(y-x)| \\
		\!= \!|\int_{0}^{1}  \left(\nabla f(x \!+\!\lambda(y \!-\! x)) \!-\! \nabla f(x) \!-\! \lambda \nabla^2 f(x)(y \!-\! x)\right)\!^{\top}\!(y \!- \! x) d\lambda|\\
		{\leq\frac{1}{2}M||y-x||^{3}\int_{0}^{1}\lambda^{2}d\lambda=\frac{M}{6}\|y-x\|^{3}.}\\ 
	\end{array} 
	$$
\end{proof}
Lemma \ref{lem:cov1} states that for functions satisfying the second-order Lipschitz continuity condition, the fitting errors of approximating them with a quadratic function is finite, and it provides an upper bound for the fitting errors.
\begin{lemma} \cite{nesterovCubicRegularizationNewton2006}
	\label{lem:cov2}
	If a twice differentiable function $ f(x) $, and the  gradient of function $ f $ is Lipschitz continuous on $\mathcal{F}$:
	$$	\|\nabla f(x)-\nabla f(y)\| \leq M\|x-y\|, \forall x,y \in \mathcal{F} $$
	for $M>0$, then for any $x$ and $y$ on $\mathcal{F}$:
	$ \| \nabla^2f(x) \| \leq  M $
	$$ |(x-y)^{\top} \nabla^2 f(y)(x-y) \leq M(x-y)^{\top}(x-y)| $$	
\end{lemma}
\begin{proof}
	$$ \| \nabla f\left(x +\lambda(y-x)\right) -  \nabla f(x)\| \leq  M \|\lambda(x-y)\|  $$
	$$ \dfrac{\| \nabla f\left(x +\lambda(y-x)\right) -  \nabla f(x)\|}{\|\lambda(x-y)\|} \leq  M $$
	$$ \| \nabla^2f(x) \| = \lim\limits_{\lambda \rightarrow 0} \dfrac{\| \nabla f\left(x +\lambda(y-x)\right) -  \nabla f(x)\|}{\|\lambda(x-y)\|} \leq  M $$
	Therefore, 
	$
	| (y-x)^{\top} \nabla^2 f(x)(y-x) | \leq   M(y-x)^{\top}(y-x) .
	$
\end{proof}

\begin{assumption}
	\label{ass:3}
	For the functional form of $ \pi_{\text {approx }}\left(x ; \theta\right) $, there are exists a set $\Theta $. When $\theta \in \Theta $ the maximum fitting errors does not exceed $\epsilon$,and $\epsilon \le  \rho$ that is 
	$$
	\max(\left\|\pi_{\text {approx }}\left(x ; \theta\right)-u^*\right\|^2) \leq \epsilon \leq  \rho.
	$$
\end{assumption}
Under Assumption~\ref{ass:3}, the function approximator has sufficient approximation capability to ensure the fitting errors is less than $\rho$.
Otherwise, the active constraints may be changed. In that case, using $\Delta{\mathcal{L}}(u,p)$ to measure control law performance would be inappropriate.

\begin{theorem}
	\label{the:lw}
	Given Assumptions \ref{ass:lH} and \ref{ass:3},	let $\mathbf{s}^*(p_i)$ be the optimal solution obtained by solving $\mathcal{P}(p_i)$ and let  $u^*_i$ be the optimal control input , $u \in \mathbf{S}(\mathbf{s}^*(p_i),\epsilon) $, then
	\begin{enumerate}
		\item $\Delta{\hat{\mathcal{L}}}_{p_i}(u)$  is a quadratic approximation to  $\Delta{\mathcal{L}}_{p_i}(u)$, and its error upper bound is $\frac{M_{\text{H}}(p_i)}{6}\|u- u_i^*\|^3$.
		\item Minimizing $L_{\text{w}}(\theta)$ is 	
		equivalent to minimizing the upper bound of $L_{\text{lag}}(\theta)$, and the upper bound is $L_{\text{w}}+ \frac{1}{6}M_{\text{Hs}}\epsilon^3$, where $M_{\text{Hs}} =\sup_{i=1,2,\dots,N_s} \left\{ M_{\text{H}}(p_i) \right\}$.
		($M_{\text{H}}$ is defined in Assumption \ref{ass:lH})
	\end{enumerate}	 	
\end{theorem}
\begin{proof}
	In Assumption \ref{ass:lH}, $\|\nabla^2 \mathcal{L}(x,p)-\nabla^2 \mathcal{L}(y,p)\|\leq M_{\text{H}}(p)\|x-y\|$.
	According to Lemma~\ref{lem:cov1}, for each $p_i$,
	$$
	\begin{aligned}
		|	\Delta \mathcal{L}(u) - (u- u_i^*)^{\top}{\mathcal{L}_{uu}(p_i)}(u- u_i^*) | &\leq \frac{M_{\text{H}}(p_i)}{6}\|u- u_i^*\|^3\\   
	\end{aligned},
	$$ 
	$$
	\begin{aligned}
		|	\Delta \mathcal{L}(u) | &\leq | (u- u_i^*)^{\top}{\mathcal{L}_{uu}(p_i)}(u- u_i^*) | +   \frac{M_{\text{H}}(p_i)}{6}\|u- u_i^*\|^3\\
		& \leq   | (u- u_i^*)^{\top}{\mathcal{L}_{uu}(p_i)}(u- u_i^*) | +\frac{\epsilon^3}{6} M_{\text{Hs}}   
	\end{aligned},
	$$ 
	
	Under Assumptions \ref{ass:3}, for each $x_i \in \mathcal{X}_{\text{feas}}$, $\pi_{\text{approx}}(x_i,\theta) \in \mathbf{S}(\mathbf{s}^*(p_i),\epsilon) $, then
	$$
	\Delta \mathcal{L}(\pi_{\text{approx}}(x_i,\theta)) \leq \Delta \hat{\mathcal{L}}(\pi_{\text{approx}}(x_i,\theta))+\frac{\epsilon^3}{6} M_{\text{Hs}},
	$$
	so 
	$
	L_{\text{lag}} \leq L_{\text{w}}+ \frac{1}{6}M_{\text{Hs}}\epsilon^3.
	$
\end{proof}
The first part of Theorem \ref{the:lw} shows that, at each optimal point $u_i^*$, using \eqref{eq:2lag} to approximate \eqref{eq:1lag} locally, an upper bound for the fitting errors can be obtained.
The second part of  Theorem \ref{the:lw} indicates that when Assumption \ref{ass:3} holds, there exists a fixed upper bound for the entire operational region. Therefore, it can also be concluded that there is an upper bound for approximating $L_{\text{lag}}$ using $L_{\text{w}}$. 
In summary, Theorem \ref{the:lw} shows that employing cost guided learning can give a guaranteed performance for the approximation policy. 

Next, we will relate the error-guided loss $L_{\text{mse}} $to the cost-guided loss $L_{\text{w}}$ through an additional looser upper bound. The theorem is stated as follows:
\begin{theorem}
	\label{the:er}
	Given Assumptions \ref{ass:lJ} and \ref{ass:3},	let $\mathbf{s}^*(p_i)$ be the optimal solution obtained by solving $\mathcal{P}(p_i)$ and let  $u^*_i$ be the optimal control input, $u_i \in \mathbf{S}(\mathbf{s}^*(p_i),\epsilon) $, then
	\begin{enumerate}
		\item 	$M_{\text{J}}(\mathbf{s}^*(p_i))(u- u_i^*)^{\top}(u- u_i^*)$  is an upper bound to  $\Delta{\hat{\mathcal{L}}}(u)$.
		\item Minimizing $L_{\text{mse}}(\theta)$ is equivalent to minimizing the upper bound of $L_{\text{w}}(\theta)$, and the upper bound is $M_{\text{Js}}L_{\text{mse}}$, where  $M_{\text{Js}}=\sup_{i=1,2,\dots,N_s} \left\{ M_{\text{J}}(\mathbf{s}^*(p_i)) \right\}$.
	\end{enumerate}	
\end{theorem}
\begin{proof}
In Assumption \ref{ass:lJ}, $\|\nabla \mathcal{L}(x,p)-\nabla \mathcal{L}(y,p)\|\leq M_{\text{J}}(p)\|x-y\|$.
According to Lemma~\ref{lem:cov2}, for any $u \in \mathbf{S}(u_i^*,\epsilon)$,  we have 
$
\| \nabla^2_u \mathcal{L}(u,p) \| \leq M_{\text{J}}(p_i)  ,
$
then
$$\begin{aligned}
	|(u- u_i^*)^{\top}{\mathcal{L}_{uu}(p_i)}(u- u_i^*) |
	\leq M_{\text{J}}(p_i) \|u- u_i^*\|^2 \leq \|u- u_i^*\|^2.
\end{aligned}$$

So under Assumptions \ref{ass:3}, for each $x_i \in \mathcal{X}_{\text{feas}}$, $\pi_{\text{approx}}(x_i,\theta) \in \mathbf{S}(\mathbf{s}^*(p_i),\epsilon) $, then
$$
\begin{aligned}
	&\Delta \hat{\mathcal{L}}(\pi_{\text{approx}}(x_i,\theta))
	\leq  M_{\text{Js}}\|\pi_{\text{approx}}(x_i,\theta)- u_i^*\|^2,
\end{aligned}
$$
so 
$
L_{\text{w}} \leq  M_{\text{Js}}L_{\text{mse}}
$
\end{proof}
Thus, combining Theorem~\ref{the:lw} and Theorem~\ref{the:er}, the following corollary can be given.
\begin{corollary}
	\label{cor:leq}
	Given the same assumptions in Theorem~\ref{the:lw}, then minimizing $L_{\text{mse}}(\theta)$ is equivalent to minimizing the upper bound of $L_{\text{lag}}(\theta)$, and the upper bound is $M_{\text{Js}}L_{\text{mse}}+ \frac{1}{6}M_{\text{Hs}}\epsilon^3$.
\end{corollary}
We introduce $\pi_{\text{approx}}(x,\hat{\theta}_{\text{mse}})$ and $\pi_{\text{approx}}(x,\hat{\theta}_{\text{w}})$ to represent the approximate policies obtained by error guided learning and cost guided learning respectively.
According to Theorem~\ref{the:lw} and Corollary~\ref{cor:leq}, both $\pi_{\text{approx}}(x,\hat{\theta}_{\text{mse}})$ and $\pi_{\text{approx}}(x,\hat{\theta}_{\text{w}})$ can guarantee that the loss generated by implementing the approximate policy is limited. Smaller fitting error yields smaller upper bound of loss. 
Furthermore, based on Theorem~\ref{the:er}, $\pi_{\text{approx}}(x,\hat{\theta}_{\text{mse}}$) can offer a smaller guaranteed loss cap.

It is worth to note that  $\mathcal{L}_{uu}(p_i)$ is constant in linear MPC problems (a linear time invariant system and a constant quadratic cost).
Thus, concerning linear MPC policy approximation, the optimization of the three functions, $L_{\text{w}}$, $L_{\text{mse}}$, and $L_{\text{lag}}$, is equivalent.

\section{Illustrative Examples}
We employ the proposed methodology to tackle a benchmark problem in the context of Model Predictive Control (MPC) policy approximation\cite{hertneck2018a}. 
This benchmark problem involves a Continuous Stirred-Tank Reactor (CSTR) with two states: the scaled concentration denoted by $x_1$ and the reactor temperature denoted by $x_2$. The coolant flow rate, $u$, is used for controlling the process. The model is described as follows:
$$
\begin{aligned}
	\dot{x}_1 &= \frac{1}{\tau}(1-x_1) - kx_1e^{-\beta/x_2} \\
	\dot{x}_2 &= \frac{1}{\tau}(x_f - x_2) + kx_1e^{-\beta/x_2} - \alpha u(x_2-x_c)
\end{aligned}
$$
The model parameters are: $\tau=20$, $k=300$, $\beta=5$, $x_f=0.3947$, $x_c=0.3816$, and $\alpha=0.117$. The constraints are ${\mathcal{X}}= \begin{bmatrix} 0.0632 & 0.4632 \end{bmatrix} \times \begin{bmatrix}  0.4519 & 0.8519 \end{bmatrix}$  and ${\mathcal{U}}=[0,2]$.
The desired setpoint is given by $x^{sp}=\begin{bmatrix} 0.2632 & 0.6519 \end{bmatrix}^T$ which is an unstable static state and the steady-state input $u_e=0.7853$.
The stage cost is defined as:
$$
\ell(x,u)=\|x-x^{s p}\|^{2}+10^{-4}\|u-u_e\|^{2}. 
$$
The MPC problem is solved with a sampling time of 1 seconds and a prediction horizon of $N=140$.
In this paper, we use a grid-based sampling approach to generate the learning samples in a similar way of \cite{hertneck2018a}, where the training data is evaluated at each grid point.

To approximate the MPC policy, we generate $N_s = 268$ state-action data pairs using a grid-based sampling approach with an interval of 0.02 for both $x_1$ and $x_2$. 
In each of the following scenarios, we also solve the complete NLP problem to compute $\pi_{\mathrm{MPC}}(p)$ for each sample, which serves as a benchmark. 
The NLP problems $\mathcal{P}(p)$ are solved using IPOPT \cite{wachter2006}  with the MUMPS linear solver. For formulating the optimization problem and the sensitivities, CasADi \cite{andersson2019} v3.6.4 was employed. 
All computations were executed on a machine with a 2.60 GHz processor and 16GB of memory.

\begin{figure}[t]
	\centering
	\includegraphics[width=0.9\linewidth]{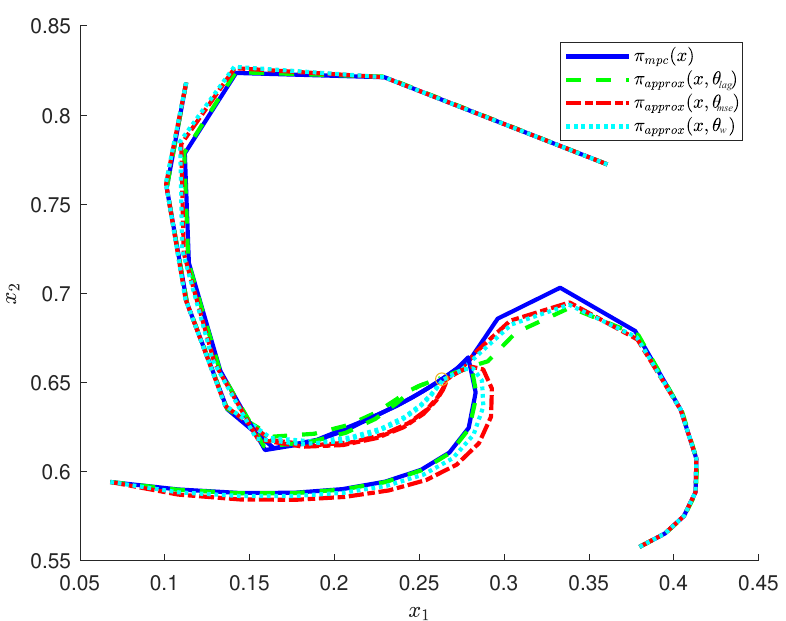}
	\caption{State trajectories starting from different initial conditions given by the MPC policy  (blue) and the approximate policies trained using $L_{\text{lag}}$ (green), $L_{\text{mse}}$ and $L_{\text{w}}$  }
	\label{fig:trajectories}
\end{figure}
Using the generated datasets, we approximate the MPC policy using deep neural networks with 3 hidden layers and 5 neurons in each hidden layer with the \texttt{thanh} as the activation function in each neuron. 
First we approximate the MPC policy under the framework demonstrated in \eqref{eq:mse}, which gives the control policy $\pi_{\text{approx}}(x; \hat{\theta}_{\text{mse}})$. Then we approximate the MPC policy under the framework demonstrated in Algorithm~\ref{alg:ampcp_cg}, which gives the control policy $\pi_{\text{approx}}(x; \hat{\theta}_{\text{w}})$. And we also use \eqref{eq:1lag} as the loss function to train the policy, which gives the control policy $\pi_{\text{approx}}(x; \hat{\theta}_{\text{lag}})$. 
When training the neural network, we use casadi's automatic differentiation tool to solve the derivative, update it with the Adma algorithm, and iterate 2500 times



We evaluate the closed-loop performance of each policy in 35 scenarios. In each scenario, the initial states are randomly sampled from $\mathcal{X}_{\text{feas}}$ (assuming $x_1$ and $x_2$ follow a uniform distribution). Subsequently, in each scenario, the control policy will be executed at intervals of 1 seconds, for a total of 100 executions. Here, $\pi_{\text{MPC}}$ served as the baseline. 

Fig.~\ref{fig:trajectories} shows the state trajectories from different initial conditions given by the MPC policy $\pi_{\text{mpc}}(x)$, the approximate policy $\pi_{\text{approx}}(x,\hat{\theta}_{\text{mse}})$ trained using error guided learning approach and the approximate policy $\pi_{\text{approx}}(x,\hat{\theta}_{\text{w}})$ trained using cost guided learning approach and the Lagrangian based policy $\pi_{\text{approx}}(x,\hat{\theta}_{\text{lag}})$.



Next, we compare the closed-loop performance of approximated policies under different training strategies. We compared their offline training times and online inference times. The results are presented in Table~\ref{tab:loss}.

\begin{table}[h]
	\centering
	\caption{Performance of policies}
	\label{tab:loss}
	\begin{tabular}{@{}cccccc@{}}
		\toprule
		&      & MPC    & $\hat{\theta}_{\text{lag}}$& $ \hat{\theta}_{\text{mse}}$ & $\hat{\theta}_{\text{w}}$ \\ \midrule
		\multicolumn{2}{c}{Offline training time(ms)}  & -      & 389.33  & 42.04    & 49.48            \\
		\hline
		Online  & mean & 197.0 & 27.72 & 27.76  & \textbf{27.58}     \\
		inference time (ms)& max  & 519.9 & 99.63  & \textbf{22.12}   & 73.46   \\
		\hline
		\multirow{2}{*}{$J-J^*$ ($10^{-3}$)}                  & mean & -      & \textbf{1.208} & 1.951   & 1.626   \\
		& max  & -      & \textbf{6.252}  & 9.017 & 10.70    \\ \cmidrule(l){1-6} 
	\end{tabular}
\end{table}
The results in Table~\ref{tab:loss} show using the Lagrangian as the loss function has the best average performance. The cost-guided learning approach is second best. For offline training time, error-guided learning and cost-guided learning are similar, but using the Lagrangian loss required much longer training time. Regarding online inference time, the approximate MPC control laws show significant advantage over MPC. These computations were done on a laptop. Considering embedded platforms with limited computing power, the advantage would be even more pronounced.



\section{Conclusion}
This paper proposed a novel cost-guided learning approach for MPC policy approximation that directly minimizes operational cost rather than just imitation accuracy. The method overcomes limitations of traditional error-guided techniques and links the fitting error with the actual operational cost. The core idea of connecting the loss function with control objectives provides a range of opportunities for future work on performance-driven learning-based control {\em e.g.} offline reinforcement learning \cite{levine2020}, inverse reinforcement learning \cite{arora2021} and virtual reference feedback tuning \cite{campi2002}. 
There are still some challenges remain for future work:
\begin{enumerate}
	\item As MPC problem size increases or more uncertainties are considered, much more data will be needed to train approximate policies. This may limit applications where data collection is difficult or expensive. Some recent work has begun addressing this issue \cite{krishnamoorthy2021}.
	\item  The cost-guided learning approach requires a significant amount of data to train the neural network approximator effectively.
	This work focused on approximating the first step of the MPC plan. Approximating closed-loop MPC policies can improve robustness\cite{kothare1996b} but poses additional difficulties. 
	\item Real-world systems involve noise, uncertainties, and sensor errors which can degrade policy performance. Combining the proposed method with techniques like stochastic MPC is an area for future work. The proposed cost-guided learning concept connects well to self-optimizing control \cite{zhou2024}, suggesting extensions of this work to handle persistent disturbances.
\end{enumerate}

%
\bibliographystyle{unsrt}
\bibliography{reference}

\end{document}